\documentclass[11pt,reqno]{amsart}
\usepackage{amscd,amssymb,amsmath,amsthm,amsfonts}
\usepackage[T1]{fontenc}
\usepackage[english]{babel}
\usepackage[arrow,matrix]{xy}
\usepackage{graphicx}
\usepackage{amsaddr}
\usepackage{orcidlink}
\usepackage{cite}
\usepackage[a4paper,top=3cm, bottom=3cm, left=3cm, right=3cm]{geometry}

\newtheorem{theorem}[subsection]{Theorem}
\newtheorem{lemma}[subsection]{Lemma}
\newtheorem{proposition}[subsection]{Proposition}
\newtheorem{corollary}[subsection]{Corollary}

\newtheorem{definition}[subsection]{Definition}

\numberwithin{equation}{section} 

\newcommand{\bea}{\begin{eqnarray}}
\newcommand{\eea}{\end{eqnarray}}

\def \> {\Rightarrow}
\def \0 {\emptyset}

\DeclareMathOperator{\inte}{int}

\DeclareMathOperator{\lcm}{lcm}
\DeclareMathOperator{\supp}{supp}
\DeclareMathOperator{\ord}{ord}
\DeclareMathOperator{\Per}{Per}
\begin{document}

\title[On a family of non-Volterra quadratic operators acting on a simplex]{On a family of non-Volterra quadratic operators acting on a simplex}

\author{Uygun Jamilov  \orcidlink{0000-0002-8860-2844},  Manuel Ladra \orcidlink{0000-0002-0543-4508}}

 \address{ U.\ U.\ Jamilov\\ V.I. Romanovskiy Institute of Mathematics,  Uzbekistan Academy of Sciences,
81, Mirzo-Ulugbek str., 100170, Tashkent, Uzbekistan.}
\email{jamilovu@yandex.ru}

\address{M.\ Ladra\\ University of Santiago de Compostela, 15782, Santiago de Compostela, Spain.}
 \email{manuel.ladra@usc.es}

\begin{abstract}
In the present paper, we consider a convex combination of non-Volterra quadratic
stochastic operators defined on a finite-dimensional simplex depending on a parameter
$\alpha$ and study their trajectory behaviors. We showed that for any $\alpha\in [0,1)$
the trajectories of such operator converge to a fixed point. For $\alpha=1$
any trajectory of the operator converges  to a periodic trajectory.
\end{abstract}

\subjclass[2010] {Primary 37N25, Secondary 92D10.}

\keywords{quadratic stochastic operator, Volterra and non-Volterra operator.}

\maketitle

\section{Introduction}

The evolution of a population can be studied by a dynamical system of a quadratic
stochastic operator \cite{L}. Such evolution operators frequently arise in many models of
mathematical genetics, namely theory of heredity (see e.g. \cite{Ber,GGJ,GSN,GZ,Gdan,Gsb,GMR,Jdnc,Jlob,Jjph,JL,JLM,K,L,RZhmn,RZhukr,ZhRsb}).

Let $E =\{1,\dots,m\}$ be a finite set and the set of all probability distributions on $E$
\begin{equation*}
S^{m-1} =\{\mathbf{x} = (x_1, x_2,\dots,x_m)\in \mathbb{R}^m : \ x_i \geq 0, \text{ for any } i  \, \text { and } \sum_{i=1}^m x_i=1\},
\end{equation*}
 the $(m-1)$-dimensional simplex.

A \emph{quadratic stochastic operator} (QSO) is a mapping $V \colon S^{m-1}\to S^{m-1}$ of the simplex
into itself, of the form $V(\mathbf{x})=\mathbf{x}'\in S^{m-1}$, where
\begin{equation}\label{qso}
V: x'_k=\sum_{i,j=1}^m p_{ij,k}x_ix_j, \quad k=1,\dots,m
\end{equation}
and the coefficients $p_{ij,k}$ satisfy
\begin{equation}\label{qsoh}
p_{ij,k}=p_{ji,k}\geq0, \quad  \sum_{k=1}^m p_{ij,k}=1, \quad i,j,k\in E.
\end{equation}

The trajectory $\{\mathbf{x}^{(n)}\}, n = 0,1,2,\dots$, of $V$ for an initial point $\mathbf{x}^{(0)}\in S^{m-1}$
is defined by
\[\mathbf{x}^{(n+1)} = V \big(\mathbf{x}^{(n)}\big)= V^{n+1}\big(\mathbf{x}^{(0)}\big), \ \ n = 0,1,2,\dots\]
Denote by $\omega_V \big(\mathbf{x}^{(0)}\big)$ the set of limit points of the trajectory $\{\mathbf{x}^{(n)}\}_{n = 0}^\infty$.

The main problem in mathematical biology consists of the study of the asymptotical
behavior of the trajectories for a given QSO. In other words, the main task is the description
of the set $\omega_V \big(\mathbf{x}^{(0)}\big)$ for any initial point $\mathbf{x}^{(0)}\in S^{m-1}$ for a given QSO. This problem
is an open problem even in two-dimensional case.

A QSO $V$ is called \emph{regular} if there is the limit $\lim\limits_{n\rightarrow\infty} V^n(\mathbf{x})$
 for any initial $\mathbf{x}\in S^{m-1}$.

A QSO $V$ is said to be \emph{ergodic} if the limit
\[\lim\limits_{n\rightarrow\infty} \frac{1}{n} \sum_{k=0}^{n-1} V^k(\mathbf{x})\]
exists for any $\mathbf{x}\in S^{m-1}$.

A \emph{Volterra} QSO is defined by \eqref{qso}, \eqref{qsoh} and with the additional
assumption
\begin{equation}\label{cv}
p_{ij,k} = 0 \ \ \text{if} \ \  k\notin \{i,j\},  \ \ i,j,k \in E.
\end{equation}
The biological treatment of conditions \eqref{cv} is rather precise: the offspring repeats the
genotype of one of its parents.

 Based on numerical calculations, Ulam conjectured that any QSO is ergodic\cite{U}.
 But in 1977 in \cite{Zus}, Zakharevich considered the following  Volterra QSO on $S^2$
\begin{equation}\label{opz}
x'_1=x_1^2+2x_1x_2, \quad x'_2=x_2^2+2x_2x_3, \quad x'_3=x_3^2+2x_1x_3
\end{equation}
and showed that it is a non-ergodic transformation, that is he proved that Ulam's conjecture
is false in general. Note that the Zakharevich's QSO \eqref{opz}  is also known as
the Stein Ulam Spiral map in some references.

Later in \cite{GZ}, the Zakharevich's result was generalized to a class of Volterra QSOs defined on $S^2$.
 In \cite{Smn}, the results of \cite{GZ}, \cite{Zus} were generalized to a class of Lotka-Volterra operators.
 In \cite{GGJ}, we  have shown the correlation between
non-ergodicity of Volterra QSOs and rock-paper-scissors games, and
in \cite{JSW}, the random dynamics of Volterra QSOs was studied.

For a recent review on the theory of quadratic stochastic operators see \cite{GMR}.

 In the present paper, we consider a specific family of discrete-time dynamical systems
generated by a convex combination of quadratic stochastic operators. The article is organized as follows.
In Section~\ref{sec:preliminaries}, we recall the definitions and known results related to a
convex combination of quadratic stochastic operators.
In Section~\ref{sec:regularoperator}, we study the asymptotical behavior of  the trajectories of a non-Volterra QSO defined on
a finite-dimensional simplex. It is shown that for an arbitrary initial point, the trajectory converges to the center
of the simplex.
In Section~\ref{sec:periodicoperator}, we consider a quasi strictly non-Volterra QSO
on a finite-dimensional simplex. We showed that almost (w.r.t. Lebesgue measure) all trajectories converge
to a periodic trajectory.
In Section~\ref{sec:convexcombination},
we consider a convex combination of QSOs which are studied in the previous two sections, study its basic properties, and therein we show that almost (w.r.t. Lebesgue measure) all trajectories convergent.

\section{Preliminaries and known results}
\label{sec:preliminaries}

Let $V$ be a quadratic stochastic operator.
A point $\mathbf{x}\in S^{m-1}$ is called a \emph{periodic} point of $V$ if there exists
an $n$  such that $V^n (\mathbf{x}) = \mathbf{x}$. The smallest positive integer $n$ satisfying the above is called
the \emph{prime period} or \emph{least period} of the point $\mathbf{x}$. A period-one point  is called a \emph{fixed} point of  $V$.
Denote the set of all periodic points (not necessarily prime) of period $n$ by $\Per_n(V)$.

Let $D V (\mathbf{x}^*) = \big(\partial V_i/\partial x_j(\mathbf{x}^*)\big)_{i,j=1}^{\ \ m}$ be the Jacobi matrix of operator $V$ at the point $\mathbf{x}^*$.

A fixed point $\mathbf{x}^*$ is called \emph{hyperbolic} if its Jacobi matrix $D V (\mathbf{x}^*)$ has no eigenvalues 1
in absolute value. A hyperbolic fixed point $\mathbf{x}^*$ is called \emph{attracting} (resp. \emph{repelling}) if all
the eigenvalues of the Jacobi matrix $D V (\mathbf{x}^*)$ are less (resp. greater) than 1 in absolute
value; it is called a  \emph{saddle} if some of the eigenvalues of $D V (\mathbf{x}^*)$ are less than 1 in absolute
value and other eigenvalues are greater than 1 in absolute value (see \cite{Dev}).

A continuous function $\varphi \colon S^{m-1} \to \mathbb{R}$ is called a \emph{Lyapunov function} for an operator $V$ if
$\varphi(V (\mathbf{x})) \leq \varphi(\mathbf{x})$ for all $\mathbf{x}$ (or $\varphi(V (\mathbf{x})) \geq \varphi(\mathbf{x})$ for all $\mathbf{x}$).

Denote $\inte S^{m-1} = \{\mathbf{x}\in S^{m-1} : x_1x_2\cdots x_m > 0\}$ and $\partial S^{m-1} = S^{m-1}\setminus  \inte S^{m-1}$.

S.S. Vallander  studied in [21] a family of quadratic stochastic operators $V_\theta \colon S^2 \to S^2$ defined
by $V_\theta = \theta V_1 + (1 -\theta)V_0, \ \ \theta \in [0, 1]$, where
\[V_0(\mathbf{x}) = \big(x^2_1+ 2x_2x_3, \,  x^2_2+ 2x_1x_3, \,  x^2_3 + 2x_1x_2\big),\]
\[V_1(\mathbf{x}) = \big(x^2_1+ 2x_1x_2, \,  x^2_2+ 2x_1x_3, \,  x^2_3 + 2x_2x_3\big).\]
It is proven that the vertices $\mathbf{e}_1 = (1, 0, 0), \, \mathbf{e}_2 = (0, 1, 0), \, \mathbf{e}_3 = (0, 0, 1)$ of the simplex and the
center $\mathbf{c} = (1/3,1/3,1/3)$ are fixed points of the operator $V_\theta$, that is, they are solutions of the equation
$V_\theta(\mathbf{x}) =\mathbf{x}$. The set $S^2 \cap \{\mathbf{x}\in S^2 : x_1 = x_3\}$ is an invariant set respect to $V_\theta$. There exists a
critic value $\theta_{cr} = 3/4$ and there is a line consisting from the fixed points of $V_\theta$ whenever
$\theta = \theta_{cr}$.

\begin{theorem}[\cite{Vlen}] For the operator $V_\theta$ the following statements are true:
\begin{itemize}
\item[i)] if $0\leq\theta < \theta_{cr}$, then $\omega_{V_\theta} \big(\mathbf{x}^{(0)}\big)=\{\mathbf{c}\}$ for any $\mathbf{x}^{(0)}\in S^2$ except the vertices;
\item[ii)] if $\theta = \theta_{cr}$, then $S^2=\mathop{\cup}\limits_{c\in[-1,1]}\{\mathbf{x}\in S^2: x_1-x_3=c\}$ and
\[\omega_{V_\theta} \big(\mathbf{x}^{(0)}\big)=\Big\{\Big(\frac{1+3c+\sqrt{1+3c^2}}{6},\frac{2-\sqrt{1+3c^2}}{3}, \frac{1-3c+\sqrt{1+3c^2}}{6}\Big)\Big\}\]
for any initial $\mathbf{x}^{(0)}\in\{\mathbf{x}\in S^2: x_1-x_3=c\}$;
\item[iii)] if $\theta_{cr}<\theta\leq 1$, then
\[\omega_{V_\theta} \big(\mathbf{x}^{(0)}\big)=\begin{cases}
\{\mathbf{e}_1\},& \text{if} \quad x_1^{(0)}> x_3^{(0)},\\
\{\mathbf{c}\},& \text{if} \quad x_1^{(0)}= x_3^{(0)}>0,\\
\{\mathbf{e}_3\},& \text{if} \quad x_1^{(0)}< x_3^{(0)}.\\
\end{cases}\]
\end{itemize}
\end{theorem}

R.N. Ganikhodjaev  considered  in \cite{Gdan} a family of quadratic operators $V_\lambda \colon S^2 \to S^2$ defined
by $V_\lambda = \lambda V_2 + (1 -\lambda)V_0, \  \lambda \in [0, 1]$, where
\[V_2(\mathbf{x}) = \big(x^2_1+ 2x_1x_2, \,  x^2_2+ 2x_2x_3, \,  x^2_3 + 2x_1x_3\big).\]

\begin{theorem}[\cite{Gdan}]   For the operator $V_\lambda$ the following statements are true:
\begin{itemize}
\item[i)] the vertices $\mathbf{e}_1, \mathbf{e}_2, \mathbf{e}_3$ and the center $\mathbf{c}$ are fixed points;
\item[ii)] if $\lambda > 1/2$, then the vertices  $\mathbf{e}_1, \mathbf{e}_2, \mathbf{e}_3$ are repelling points and they are saddles
when $\lambda < 1/2$;
\item[iii)] if $\lambda > 1 -\sqrt{3}/2$, then the center $\mathbf{c}$ is an attracting point and it is a repelling point
when $\lambda < 1 -\sqrt{3}/2$.
\item[iv)] if $\lambda > 1/2$, then  $\omega_{V_\lambda} \big(\mathbf{x}^{(0)}\big)=\{\mathbf{c}\}$ for any $\mathbf{x}^{(0)}\in S^2\setminus\{\mathbf{e}_1, \mathbf{e}_2, \mathbf{e}_3\}$;
\item[v)] if $0 < \lambda < 1 -\sqrt{3}/2$, then $\omega_{V_\lambda} \big(\mathbf{x}^{(0)}\big)\subset \inte  S^2 $ is an infinite set for any
$\mathbf{x}^{(0)}\in \inte S^2\setminus \{\mathbf{c}\}$.
\end{itemize}
\end{theorem}

S.S. Vallander  studied in \cite{Vsp} a family of quadratic stochastic operators on $S^2$ defined by
$V_\lambda = \lambda V_2 + (1 -\lambda)V_3, \  \lambda \in [0, 1]$,  where
\[V_3(\mathbf{x}) = \big(x^2_1+ 2x_1x_3, \,  x^2_2+ 2x_1x_2, \,  x^2_3 + 2x_2x_3\big).\]
It is evident that $V_\lambda$ is the identity map whenever $\lambda = 1/2$. The trajectories under
consideration are (discrete) \emph{spirals} with infinitely many coils approaching $\partial S^2$ \cite{Vdan}.

\begin{theorem} [\cite{Vsp}] For the operator $V_\lambda$ the following statements are true:
\begin{itemize}
\item[i)] if $\lambda \neq 1/2$, then the function $\varphi(\mathbf{x}) = x_1x_2x_3$ is a Lyapunov function;
\item[ii)] if $\lambda < 1/2$ (resp. $\lambda > 1/2$), then the trajectories of $V_\lambda$ are spirals similar to the
trajectories of $V_2$ (resp. $V_3$).
\end{itemize}
\end{theorem}
%It is easy to understand that, as $\lambda$ approaches 1/2, the "speed of rotation" of these
%spirals unboundedly decreases and at the point $\lambda = 1/2$, the trajectories change their
%orientation.

N. N. Ganikhodjaev et al.  studied in \cite{GSN} two families of quadratic stochastic operators
on $S^2$ defined by $V_\alpha = (1 -\alpha) V_2 +\alpha V_4,  \ \alpha \in [0, 1]$  and $V_\beta = (1 -\beta) V_2 +\beta V_5, \  \beta \in [0, 1]$,  where
\[V_4(\mathbf{x}) = \big(x^2_2+ 2x_1x_2, \,  x^2_3+ 2x_2x_3, \,  x^2_1 + 2x_1x_3\big),\]
\[V_5(\mathbf{x}) = \big(x^2_3+ 2x_1x_2, \,  x^2_1+ 2x_2x_3, \,  x^2_2 + 2x_1x_3\big).\]
\begin{theorem}[\cite{GSN}] For the operators $V_\alpha$ and $V_\beta$ the following statements are true:
\begin{itemize}
\item[i)] if $\alpha = 1/2$, then the function  $\varphi(\mathbf{x}) = |x_1-x_2||x_2-x_3||x_3-x_1|$ is a Lyapunov
function and $\omega_{V_\alpha} \big(\mathbf{x}^{(0)}\big)=\{\mathbf{c}\}$ for any $\mathbf{x}^{(0)}\in S^2$ ;
\item[ii)] if $\alpha \neq 1/2$, then $\omega_{V_\alpha} \big(\mathbf{x}^{(0)}\big)\subset \inte S^2$ is an infinite compact subset for any
$\mathbf{x}^{(0)}\in \inte S^2$ except the center $\mathbf{c}$.
\item[iii)] if $0 < \beta < 1 -\sqrt{3}/2$, then $\omega_{V_\beta} \big(\mathbf{x}^{(0)}\big)\subset \inte S^2$ is an infinite compact subset for any
$\mathbf{x}^{(0)}\in \inte S^2$ except the center $\mathbf{c}$.
\item[iv)] if $1 -\sqrt{3}/2\leq \beta < 1$, then $\omega_{V_\beta} \big(\mathbf{x}^{(0)}\big)=\{\mathbf{c}\}$ for any $\mathbf{x}^{(0)}\in S^2$;
\end{itemize}
\end{theorem}

In \cite{Jjph} we studied a family of quadratic stochastic operators on $S^2$ defined by $V_\theta = \theta V_6 + (1 -\theta)V_7, \ \theta \in [0, 1]$, where
\[V_6(\mathbf{x}) = \big(x^2_3+ 2x_2x_3, \,  x^2_1+ 2x_1x_3, \,  x^2_2 + 2x_1x_2\big),\]
\[V_7(\mathbf{x}) = \big(x^2_2+ 2x_2x_3, \,  x^2_3+ 2x_1x_3, \,  x^2_1 + 2x_1x_2\big).\]
\begin{theorem} [\cite{Jjph}] For the operator $V_\theta$ the following statements are true:
\begin{itemize}
\item[i)] if $\theta = 1/2$, then the function  $\varphi(\mathbf{x}) = |x_1-x_2||x_2-x_3||x_3-x_1|$ is a Lyapunov
function and $\omega_{V_\theta} \big(\mathbf{x}^{(0)}\big)=\{\mathbf{c}\}$ for any $\mathbf{x}^{(0)}\in S^2$ ;
\item[ii)] if $\theta \neq 1/2$, then $\omega_{V_\theta} \big(\mathbf{x}^{(0)}\big)\subset \inte S^2$ is an infinite compact subset for any
$\mathbf{x}^{(0)}\in \inte S^2$ except the center $\mathbf{c}$.
\end{itemize}
\end{theorem}

\section{A regular non-Volterra quadratic stochastic operator}
\label{sec:regularoperator}

In this section, we consider  the following  non-Volterra QSOs defined on a finite-dimensional simplex $S^{m-1}$
\begin{equation}\label{eqso}
V: x'_k= x_k^2 +\frac{2}{m-2}\displaystyle\sum_{\substack{i,j\in E\setminus\{k\}\\ i<j}}  x_ix_j,\  \ k \in E,
\end{equation}
where $E=\{1,\dots,m\}$ and  $m\geq3$.

Since in the case $m=3$ the QSO \eqref{eqso} coincides with the QSO $V_0$ in the below
we suppose that $m>3$.

Denote by $\mathbf{e}_i = (\delta_{1i},\delta_{2i}, \dots, \delta_{mi}) \in S^{m-1}, \, i = 1,\dots,m$, the vertices of the simplex $S^{m-1}$,
where $\delta_{ij}$ is the Kronecker delta and by $\mathbf{c}=(1/m,\dots,1/m)$ the center of the simplex $S^{m-1}$.

\begin{proposition} For the QSO \eqref{eqso} the following statements are true:
\begin{itemize}
\item[i)] The vertices $\{\mathbf{e}_k\}_{k\in E}$ and the center $\mathbf{c}$ are fixed points;
\item[ii)] The vertices $\{\mathbf{e}_k\}_{k\in E}$  are non-hyperbolic points,
 and the center $\mathbf{c}$ is an attracting point.
\end{itemize}
\end{proposition}

\begin{proof} i) It easy to check that the  vertices $\mathbf{e}_1,\dots,\mathbf{e}_m$
satisfy the equation $V(\mathbf{x})=\mathbf{x}$, that is, they are fixed points.

By  using $x_1+\dots+x_m=1$
in the equation $V(\mathbf{x})=\mathbf{x}$, we consider the following difference, for $ u,v \in E, \ u\neq v$,
\[x_u-x_v=(x_u-x_v)\Big((x_u+x_v)+\frac{2}{m-2} (1-x_u-x_v)\Big).\]
Since it holds
\[(x_u+x_v)+\frac{2}{m-2} (1-x_u-x_v)>0,\]
it follows that $x_u=x_v, \  \text{for all} \  u,v \in E$.
Therefore, we obtain the center $\mathbf{c}$.

ii) After some algebra we have that for any $k\in E$ the Jacobi matrix $D V (\mathbf{e}_k)$ has
the eigenvalues $\lambda_1=2$ and $|\lambda_i|=1, \  i=2,\dots, m-2$.
The Jacobi matrix $D V (\mathbf{c})$ has the eigenvalues equal to zero.
Therefore,  the vertices are non-hyperbolic points, and the center is an attracting point.
\end{proof}

\begin{theorem}  For the QSO \eqref{eqso} the following statements are true:
\begin{itemize}
\item[i)] The function $\varphi(\mathbf{x}) =|x_1-x_2||x_2-x_3|\cdots |x_m-x_1|$ is a Lyapunov function;
\item[ii)] For any $\mathbf{x}^{(0)}\in S^{m-1}$ the trajectory of the non-Volterra QSO \eqref{eqso} converges to $\mathbf{c}$.
\end{itemize}
\end{theorem}

\begin{proof} i) From  \eqref{eqso} we have $\varphi(V(\mathbf{x})) =\varphi(\mathbf{x}) \psi(\mathbf{x})$, where
\begin{equation*}
 \psi(\mathbf{x})=\frac{1}{(m-2)^m}\big(2 +(m-4)(x_1+x_2)\big)\big(2 +(m-4)(x_2+x_3)\big)\cdots \big(2 +(m-4)(x_1+x_m)\big).
\end{equation*}

Using the AM-GM inequalities, one easily has
\begin{align}\label{estpsi}
\psi(\mathbf{x})&\leq \frac{1}{(m-2)^m} \Big(2 +\frac{m-4}{m}(x_1+x_2+x_2+x_3+\cdots+x_m+x_1)\Big)^m\notag\\
&=\frac{1}{(m-2)^m}   \Big(\frac{4(m-2)}{m}\Big)^m=\Big(\frac{4}{m}\Big)^m\leq 1.
\end{align}

Therefore, the function $\varphi(\mathbf{x})$ is a Lyapunov function for the QSO \eqref{eqso}.\\

ii)  {\it CASE} $m\geq5$ : Let $\mathbf{x}^{(0)}\in S^{m-1}\setminus \{\mathbf{e}_1,\dots,\mathbf{e}_m\}$. Then
from \eqref{estpsi} we have $\psi(\mathbf{x})<1$ and so $\varphi(V(\mathbf{x})) <\varphi(\mathbf{x})$.
Consequently, it follows
\[\varphi(\mathbf{x}^{(n+1)})\leq \Big(\frac{4}{ m}\Big)^m \varphi(\mathbf{x}^{(n)})\leq \cdots
\leq \Big(\frac{4}{m}\Big)^{m(n+1)} \varphi(\mathbf{x}^{(0)}). \]

Since $\varphi(\mathbf{x}^{(0)})<1$ from the last one has
\[\lim\limits_{n\rightarrow\infty} \varphi(\mathbf{x}^{(n)})=0.\]

Thus $\lim\limits_{n\rightarrow\infty}\mathbf{x}^{(n)}=\mathbf{c}$ for any $\mathbf{x}^{(0)}\in S^{m-1}\setminus \{\mathbf{e}_1,\dots,\mathbf{e}_m\}$.\\

{\it CASE} $m=4$ : In this case the QSO \eqref{eqso} has the following form
\begin{equation*}
\left\{\begin{array}{llll}
               x'_{1}= x_1^2+x_2x_3+x_2x_4+x_3x_4, \\[2mm]
               x'_{2}= x_2^2+ x_1x_3+x_1x_4+x_3x_4, \\[2mm]
               x'_{3}= x_3^2+ x_1x_2+x_1x_4+x_2x_4, \\[2mm]
               x'_{4}= x_4^2+ x_1x_2+x_1x_3+x_2x_3.
\end{array}\right.
\end{equation*}

We set $\|\mathbf{x}\|=\max\limits_{1\leq k\leq 4} x_k $ for $\mathbf{x}\in S^3$.
It is easy to see that for any $\mathbf{x}\in S^3\setminus \{\mathbf{e}_1,\mathbf{e}_2,\mathbf{e}_3,\mathbf{e}_4, \mathbf{c}  \}$
\[\|V(\mathbf{x})\|=\max\limits_{1\leq k\leq 4} x'_k < \|\mathbf{x}\| x_1+
\|\mathbf{x}\| x_2+\|\mathbf{x}\| x_3+\|\mathbf{x}\| x_4= \|\mathbf{x}\|.\]
Therefore, the sequence  $\{\|\mathbf{x}^{(n)}\|\}_{n=0}^\infty$ is a strong decreasing and bounded from the below.
So there is the limit $\lim\limits_{n\rightarrow\infty} \|\mathbf{x}^{(n)}\| =1/4$  for any $\mathbf{x}^{(0)}\in S^3\setminus \{\mathbf{e}_1,\mathbf{e}_2,\mathbf{e}_3,\mathbf{e}_4,\mathbf{c}\}$.
Since
\[ \min\limits_{\mathbf{x}\in S^3} \|\mathbf{x}\| = \|\mathbf{c}\|=\frac{1}{4} \ \ \text{iff} \ \ \mathbf{x}=\mathbf{c}\]
it follows that  $\lim\limits_{n\rightarrow\infty} \mathbf{x}^{(n)} =\mathbf{c}$.

Thus $\lim\limits_{n\rightarrow\infty}\mathbf{x}^{(n)}=\mathbf{c}$ for any $\mathbf{x}^{(0)}\in S^3\setminus \{\mathbf{e}_1,\dots,\mathbf{e}_4\}$.
\end{proof}

\section{Quasi strictly non-Volterra quadratic stochastic operator}
\label{sec:periodicoperator}

Recently,  the following quasi strictly non-Volterra operator on $S^2$
\[V(\mathbf{x}) = \big(x^2_1+ (x_2+x_3)^2, \ \  2x_1x_3, \ \  2x_1x_2 \big)\]
was studied in \cite{Khukr}.

\begin{theorem}[\cite{Khukr}]  For the operator $V$ the following statements are true:
\begin{itemize}
\item[i)] the vertex $\mathbf{e}_1$ and  $\overline{\mathbf{x}}=(1/2,1/4,1/4)$ are fixed points;
\item[ii)] $V^2(\mathbf{x})=\mathbf{x}$ for any $\mathbf{x}\in \{\mathbf{x}\in S^2: x_1=1/2, x_2+x_3=1/2\}$;
\item[iii)] The sets $ M_\tau=\{\mathbf{x}\in S^2: x_2=\tau x_3 \vee x_2=(1/\tau) x_3 \} $ when $\tau>0$, and
  $ M_0=\{\mathbf{x}\in S^2: x_2x_3=0 \} $  when $\tau=0$, are invariant sets of $V$;
\item[iv)] if  $\mathbf{x}^{(0)}\in \Theta=\{\mathbf{x}\in S^2: x_1=0\}\cup\{\mathbf{e}_1\}$, then $\omega_{V} \big(\mathbf{x}^{(0)}\big)=\{\mathbf{e}_1\}$;
\item[v)]  $\omega_{V} \big(\mathbf{x}^{(0)}\big)=\{(1/2,0,1/2), (1/2,1/2,0)\}$ for any initial $\mathbf{x}^{(0)}\in M_0$;
\item[vi)] for any $\mathbf{x}^{(0)}\in S^2 \setminus\Theta$ there is $M_\tau,  \ \tau\in(0,+\infty)$ such that $\mathbf{x}^{(0)}\in M_\tau$ and $\omega_{V} \big(\mathbf{x}^{(0)}\big)=\{\tilde{\mathbf{x}},\hat{\mathbf{x}}\}$, where
    \[\tilde{\mathbf{x}}=\Big(\frac{1}{2},\frac{\tau}{2(1+\tau)},\frac{1}{2(1+\tau)} \Big) \quad \text{and} \quad  \hat{\mathbf{x}}=\Big(\frac{1}{2},\frac{1}{2(1+\tau)},\frac{\tau}{2(1+\tau)} \Big).\]
\end{itemize}
\end{theorem}

In this section, we attempt to generalize the results of \cite{Khukr} for a quasi strictly non-Volterra operator defined
on a finite-dimensional simplex.

Consider the following QSO defined on a finite-dimensional simplex $S^{m-1}$ in the form
\begin{equation}\label{op1}
V_\pi: \left\{\begin{array}{ll}
              x'_{k}= 2x_mx_{\pi(k)},\qquad 1\leq k\leq m-1 \\[2mm]
              x'_m=x^2_m+ (x_1+\dots+x_{m-1})^2 ,
\end{array}\right.
\end{equation}
where $\pi$ is a permutation of the set $\{1,\dots,m-1\}$.

It is evident that $x'_m=f(x_m)$, where
\[ f(x)=2x^2-2x+1, \ \ x\in [0, 1].\]

We denote $f^n(x)=\underbrace{f\circ\cdots\circ f}_{n \, \text{times}} (x)$,  the $n$-fold composition of $f(x)$ with itself.

\begin{proposition}[\cite{Khukr}]\label{per2f} There is no periodic orbit of $f(x)$ with period $n>2$.
\end{proposition}

\begin{proposition}\label{trajf}
If $0<x<1$, then $\lim\limits_{n\rightarrow\infty} f^n(x) =1/2$.
\end{proposition}

\begin{proof}
The proof of this lemma follows from Lemma \ref{falfa} substituting $\alpha=0$.
\end{proof}

 A permutation $\pi\in S_n$ is a $k$--\emph{cycle} if there exists a positive integer $k$
and an integer $i\in \{1,\dots, n\}$ such that
\begin{itemize}
\item[(1)] $k$ is the smallest positive integer such that $\pi^k (i) = i$, and
\item[(2)] $\pi$ fixes each $j\in \{1, \dots, n\} \setminus \{i, \pi(i), \dots , \pi^{k-1}(i)\}$.
\end{itemize}
The $k$-cycle $\pi$ is usually denoted $\big(i, \pi(i), \dots , \pi^{k-1}(i)\big)$.

Every permutation can be represented in the form of a product of cycles without common elements (i.e. disjoint cycles) and this representation is unique to within the order of the factors.

Let $\pi=\tau_1\tau_2\cdots\tau_q$ be a permutation of the set $\{1,\dots,m-1\}$, where $\tau_1,\dots,\tau_q$ are disjoint cycles and
we denote by $\ord (\tau_i)$ the order of a cycle $\tau_i$.
Denote $\supp (\mathbf{x})=\{i: \, x_i>0\}$ and let $|\supp (\mathbf{x})|$ be its cardinality.
\begin{proposition}
For the operator $V_\pi$ the following statements are true:
\begin{itemize}
\item[i)] The  set $M_{0}=\{\mathbf{x}\in S^{m-1}: x_{1}\cdots x_{m-1}=0\}$  and
 \[M_{\omega}=\Big\{\mathbf{x}\in S^{m-1}: \, \prod_{k\in \tau_i} x_{k}= \omega \, \prod_{k\in \tau_j} x_k \,
 \vee \, \prod_{k\in \tau_i} x_{k} = (1/\omega) \, \prod_{k\in \tau_j} x_k  \Big\}\]
are invariant sets, where $ \omega>0$;
\item[ii)] The vertex $\mathbf{e}_m$ and the point $\mathbf{x}^*=(x_1^*,\dots,x_m^*)$ are fixed points, where $\mathbf{e}_m=(0,\dots,0,1)$ and
\[ x^*_i =\begin{cases}
\frac{1}{ 2},& \ \ \text{if} \ \ i=m,\\
\frac{1}{ 2(|\supp (\mathbf{x})|-1)},& \ \ \text{if} \ \ i\in \supp (\mathbf{x})\setminus \{m\},\\
0, & \ \ \text{if} \ \ i\notin \supp (\mathbf{x})\setminus \{m\}.
\end{cases}\]
%\item[iii)] The vertex $\mathbf{e}_m$ is a repelling point and the point $\mathbf{x}^*$ is a non-hyperbolic point.
\item[iii)] The points of the set $\Per_s(V_\pi) =\{\mathbf{x}\in S^{m-1} : x_1+\dots+x_{m-1}=1/2, \, x_m=1/2\}$ are
periodic points of $V_\pi$ with period $s=\lcm \big(\ord (\tau_1),\dots, \ord (\tau_q)\big)$.
\item[iv)] $ \Per_n(V_\pi) =\{\mathbf{x}\in S^{m-1} : V_\pi^n(\mathbf{x})=\mathbf{x} \}=\emptyset$ for $n>s$.
\end{itemize}
\end{proposition}

\begin{proof} i) Let $\mathbf{x} \in M_0$. Then from \eqref{op1} one has that
\[x'_1\cdots x'_{m-1} = (2x_m)^{m-1} x_{1}\cdots x_{m-1}=0,\]
that is, $V_\pi(M_0)\subset M_0$.\\

Let $\mathbf{x} \in M_\omega$ and $\tau_i,\tau_j$  cycles. Then from \eqref{op1} we have either
\[\frac{\prod\limits_{k\in \tau_i} x'_{k}}{ \prod\limits_{k\in \tau_j} x'_{k}}=
\frac{\prod\limits_{k\in \tau_i} x_{k}}{ \prod\limits_{k\in \tau_j} x_{k}}=\omega \qquad \text{or} \qquad \frac{\prod\limits_{k\in \tau_i} x'_{k}}{ \prod\limits_{k\in \tau_j} x'_{k}}=
\frac{\prod\limits_{k\in \tau_i} x_{k}}{ \prod\limits_{k\in \tau_j} x_{k}}=\frac{1}{\omega}. \]
Consequently  $V_\pi(M_\omega)\subset M_\omega$.

ii) The equation  $V_\pi(\mathbf{x})=\mathbf{x}$ has the following form
\begin{equation}\label{eqop1}
\left\{\begin{array}{llllll}
                x_{k}\,= 2x_mx_{\pi(k)}, \qquad 1\leq k\leq m-1\\[2mm]
                x_{m}= 2x^2_m-2x_m+1,
\end{array}\right.
\end{equation}

It is easy to check that the vertex $\mathbf{e}_m=(0,\dots,0,1)$  is a solution of \eqref{eqop1}.
Other solutions $\mathbf{x}^*=(x_1^*,\dots,x_m^*)$ of  system \eqref{eqop1} are defined as follows
\[ x^*_i =\begin{cases}
\frac{1}{ 2},& \ \ \text{if} \ \ i=m,\\
\frac{1}{ 2(|\supp (\mathbf{x})|-1)},& \ \ \text{if} \ \ i\in \supp (\mathbf{x})\setminus \{m\},\\
0, & \ \ \text{if} \ \ i\notin \supp (\mathbf{x})\setminus \{m\}.
\end{cases}\]

%iii) Simple algebra shows that the eigenvalues of the Jacoby matrix $D(\mathbf{e}_m)$
%are either $\mu_i =2$ or $\mu_i=-2$. So $|\mu_i|=2$ then it follows that the vertex $\mathbf{e}_m$ is a repelling point.
%The other fixed points can be checked in a similar manner.

iii)  For $s=\lcm \big(\ord (\tau_1),\dots, \ord (\tau_q)\big)$, from \eqref{op1}, we have
\begin{equation*}
 \left\{\begin{array}{ll}
               x^{(s)}_{k}=2^s x_m\prod\limits_{j=1}^{s-1} f^j(x_m) x_{k}, \ \ 0\leq k\leq m-1  \\[2mm]
              x^{(s)}_{m}= f^s(x_m),
\end{array}\right.
\end{equation*}

Due to Proposition \ref{per2f} the equation $f^s(x_m)=x_m$ has the  solutions
$x_m=1$ and $x_m=1/2$. It is easy to check that
\begin{equation}\label{qiy}
x_m\prod\limits_{j=1}^{s-1} f^j(x_m)\Big|_{x_m=1}=1 \quad \text{and} \quad 2^s x_m\prod\limits_{j=1}^{s-1} f^j(x_m)\Big|_{x_m=\frac{1}{2}}=1 .
\end{equation}
\\

Let $x_m=1$. Then from the equation $V_\pi^s(\mathbf{x})=\mathbf{x}$ we obtain
\[x_{k}=2^s x_{k}, \quad   k=1,\dots,m-1, \]
that is, $x_1=\cdots=x_{m-1}=0$. So, we get the fixed point $\mathbf{e}_m$.\\

Let $x_m=1/2$. Then from the equation $V_\pi^s(\mathbf{x})=\mathbf{x}$ we obtain
\[x_{k}= x_{k}, \quad   k=1,\dots,m-1, \]
that is, arbitrary $ x_1,\dots,x_{m-1}$  with the condition $x_1+\dots+x_{m-1}=1/2$.

Thus $V_\pi^s(\mathbf{x})=V_\pi^{s-1}(\mathbf{x}^1) =V_\pi^{s-2}(\mathbf{x}^2) = \dots=V_\pi(\mathbf{x}^{s-1})=\mathbf{x}$, where
\[
\mathbf{x}=(x_1,x_2,\dots, x_{m-1},1/2 ), \ \  \mathbf{x}^{j}=(x_{\pi^{j}(1)},x_{\pi^{j}(2)},\dots, x_{\pi^{j}(m-1 )},1/2 ), \ \ j=1,\dots,s-1.\]

iv) Using induction on $n$ enable to us for $V_\pi^n$  get that
\begin{equation*}
\left\{\begin{array}{ll}
             x^{(n)}_{k}= 2^n x_m\prod\limits_{j=1}^{n-1} f^j(x_m) x_{\pi^{n}(i)}, \ \ 1\leq k\leq m-1, \\[2mm]
              x^{(n)}_{m}= f^n(x_m).
\end{array}\right.
\end{equation*}
%where $a^{(n)}_i=\mathbf{1}_{n\equiv i(\text{mod} \, ord(\tau_l))}, \ \ i=0,\dots,ord(\tau_l)-1$.

  The equation $x_m=f^n(x_m)$ has the solutions $x_m=1$  and $x_m=1/2$.
 If we take  $x_m=1$, then we get the fixed point $\mathbf{e}_m$.

Let $x_m=1/2$. Then using \eqref{qiy} from the equation $V_\pi^n(\mathbf{x})=\mathbf{x}$ one has
\begin{equation}\label{deffer}
x_{k}=  x_{\pi^{n}(k)}, \quad 1\leq k \leq m-1.
\end{equation}

%Further we obtain
%\begin{align}\label{deffer}
%x_{k_1}- x_{k_2}&=  \sum\limits_{i\in \tau_l} a^{(n)}_{i} x_{\pi^{n}(i)}, \ \ k\in \tau_l, \ \ l=\overline{1,q}.\notag\\
%& \quad \quad +\dots+ a^{(n)}_{k-1}(x_{ki+3}- x_{ki+4}),\notag\\
%   & \vdots \\
%x_{k(i+1)+1}-x_{ki+2} &= a^{(n)}_{0}( x_{k(i+1)+1} -x_{ki+2})+a^{(n)}_{1}( x_{k(i+1)}-x_{k(i+1)+1})\notag\\
%& \quad \quad +\dots+ a^{(n)}_{k-1}(x_{ki+2}- x_{ki+3}), \ \ 0\leq i\leq s-1.\notag
%\end{align}

Let $n\equiv 1 \pmod {s}$. The cases $n\equiv 2\pmod {s}, \dots, n\equiv s-1 \pmod {s} $ can be considered  similarly.
Then system of equations \eqref{deffer} has the following form
\begin{align*}
x_{k} =  x_{\pi(k)},  \quad 1\leq k \leq m-1,
\end{align*}
and we obtain a fixed point such that
\[x_m=\frac{1}{2}, \ \ x_i=\frac{1}{ 2(|\supp (\mathbf{x})|-1)}, \ \ i\in \supp (\mathbf{x})\setminus \{m\}, \ \
x_i=0, \ \  i\notin \supp (\mathbf{x})\setminus \{m\}.\]

\vskip 3mm

If $n\equiv 0 \pmod {s}$  then \eqref{deffer} holds for any  sample $x_{1},\dots, x_{m-1}$ with
the condition  $x_2+\dots+x_m=1/2$ , that is, we have $\mathbf{x}\in \Per_s(V_\pi)$.

Thus, any solution of the equation $V_\pi^n(\mathbf{x})=\mathbf{x}$
is either a fixed point or  a periodic point with period $s$, that is
$\Per_n(V_\pi)=\emptyset$ for any $n\geq s+1$.
\end{proof}

\begin{proposition}\label{lyap}
For any $1\leq l\leq q$ the functions
\[\varphi_l (\mathbf{x})= \prod_{k\in \tau_l} x_{k} \quad \text{and} \quad \psi_l (\mathbf{x})= \sum_{k\in \tau_l} x_{k} \]
are Lyapunov functions for the operator $V_\pi$ defined by \eqref{op1}.
\end{proposition}

\begin{proof} Since the $\min f(x_m)= 1/2$  and  $V_\pi(\mathbf{x})_m= f(x_m)$, from beginning we can assume that $x_m\geq 1/2$.
Using this fact from  \eqref{op1} one easily has
\[\varphi_l (V_\pi(\mathbf{x}))=\prod_{k\in \tau_l} x'_{k} = (2x_m)^{\ord(\tau_l)} \prod_{k\in \tau_l} x_{k}\geq \varphi_l (\mathbf{x}).\]
Similarly, for the function $\psi_l (\mathbf{x})$ we have
\[\psi_l (V_\pi(\mathbf{x}))=\sum_{k\in \tau_l} x'_{k}=2x_m \sum_{k\in \tau_l} x_{k}\geq \psi_l (\mathbf{x}).\]

Thus, the functions
$\varphi_l (\mathbf{x})$ and  $\psi_l (\mathbf{x})$ are Lyapunov functions for the operator \eqref{op1}.

\end{proof}

\begin{corollary} For any $i_1,i_2,\dots, i_j$ and $\alpha_{i_1},\alpha_{i_2},\dots,\alpha_{i_j},\beta_{i_1},\beta_{i_2},\dots,\beta_{i_j}\in \mathbb{R}$, the functions
\[\varphi(\mathbf{x}) =\prod_{i\in \{i_1,i_2,\dots, i_j\} } \alpha_{i}\varphi_{i} (\mathbf{x})+ \sum_{i\in \{i_1,i_2,\dots, i_j\} } \alpha_{i}\varphi_{i} (\mathbf{x}), \]
  \[\psi(\mathbf{x}) = \prod_{i\in \{i_1,i_2,\dots, i_j\} } \beta_{i}\psi_{i} (\mathbf{x})+ \sum_{i\in \{i_1,i_2,\dots, i_j\} } \beta_{i}\psi_{i} (\mathbf{x}), \]
\[\widetilde{\varphi}(\mathbf{x})= \varphi(\mathbf{x})\psi(\mathbf{x}), \qquad   \widetilde{\psi}(\mathbf{x})=\varphi(\mathbf{x})+\psi(\mathbf{x})  \]
are  Lyapunov function for the operator $V_\pi$ defined by \eqref{op1}.
\end{corollary}

\begin{proof}
The proof follows immediately from Proposition \ref{lyap}.

\end{proof}

\begin{theorem} For the operator $V_\pi$ the following statements are true:
\begin{itemize}
\item[i)] $V_\pi(\mathbf{x})=\mathbf{e}_m $  for any $\mathbf{x}\in\{\mathbf{x}\in S^{m-1}: x_m=0\} \cup\{\mathbf{e}_m\}$;
\item[ii)]   $\omega_{V_\pi} \big(\mathbf{x}^{(0)}\big)=\{\mathbf{x}_\xi, \mathbf{x}_\xi^{1},\dots, \mathbf{x}_\xi^{s-1} \}$
for any $\mathbf{x}^{(0)}\in S^{m-1}\setminus \big(\{\mathbf{x}\in S^{m-1}: x_m=0\} \cup\{\mathbf{e}_m\}\big)$.
\end{itemize}
\end{theorem}

\begin{proof}  i) Straightforward.

ii) Let $\mathbf{x}^{(0)}\in S^{m-1}\setminus \big(\{\mathbf{x}\in S^{m-1}: x_m=0\} \cup\{\mathbf{e}_m\}\big)$ be an initial point. By Proposition~\ref{trajf}  we have that
\[\lim\limits_{n\rightarrow\infty} x_m^{(n)}=\lim\limits_{n\rightarrow\infty} f^n(x_m^{(0)})=\frac{1}{ 2}.\]

Consider the operator $V^s_\pi$  and using  $\min f(x_m)= 1/2$  one has

\begin{equation*}
\begin{array}{ll}
              x^{(s)}_{k}\geq  x_{k},   \ \ 1\leq k\leq m-1, \\[2mm]
              x^{(s)}_m=f^s(x_m)
\end{array}
\end{equation*}
Therefore, the sequences $\{ x^{(ns)}_{k}\}, \  1\leq k\leq m-1,  \ n=0,1,2,\dots$  are increasing
and bounded from the above. So, there are the limits
\[\lim\limits_{n\rightarrow\infty} x^{(ns)}_{k} =\xi_k, \quad   1\leq k\leq m-1.\]
Denote $\mathbf{x}_\xi=\big(\xi_1, \xi_2, \dots,\xi_{m-1}, 1/2\big)$.

Since the QSO $V_\pi$ is a continuous map we have
\[\lim\limits_{n\rightarrow\infty} x^{(ns+1)}_{k} =V_\pi(\mathbf{x}_\xi)_k, \quad   1\leq k\leq m-1,\]
and $\mathbf{x}_\xi^{1} =V(\mathbf{x}_\xi)=\big(\xi_{\pi(1)}, \xi_{\pi(2)}, \dots,\xi_{\pi_(m-1)}, 1/2\big)$.

Repeating this argument $(s-1)$ times we obtain
\[\mathbf{x}_\xi^{s-1} =V_\pi(\mathbf{x}_\xi^{s-2})=\Big(\xi_{\pi^{s-1}(1)}, \xi_{\pi^{s-1}(2)}, \dots,\xi_{\pi^{s-1}(m-1)}, \frac{1}{2}\Big)\]
and $\mathbf{x}_\xi =V_\pi(\mathbf{x}_\xi^{s-1})$.

Thus for $\mathbf{x}^{(0)}\in S^{m-1}\setminus \big(\{\mathbf{x}\in S^{m-1}: x_m=0\} \cup\{\mathbf{e}_m\}\big)$ the set
of limit points $\omega_V \big(\mathbf{x}^{(0)}\big)$ of the trajectory of the QSO \eqref{op1} has the following form
\[\omega_{V_\pi} \big(\mathbf{x}^{(0)}\big)=\{\mathbf{x}_\xi, \mathbf{x}_\xi^{1},\dots, \mathbf{x}_\xi^{s-1}\}. \]
\end{proof}

\section{Convex combination of operators}
\label{sec:convexcombination}

Let $\pi=\tau_1\tau_2\cdots\tau_q$ a permutation of the set $\{1,\dots,m-1\}$, where $\tau_1,\dots,\tau_q$ are disjoint cycles and
let $s=\lcm \big(\ord (\tau_1),\dots, \ord (\tau_q)\big)$.

Consider the family of QSOs $V_\alpha \colon S^{m-1}\to S^{m-1}$ defined by
$V_\alpha = \alpha V_1 + (1 -\alpha)V_2, \  \alpha \in [0, 1]$,  where
$V_1$ is the non-Volterra QSO \eqref{eqso},  and  $V_2$ is the quasi strictly non-Volterra operator
$V_\pi$ defined by \eqref{op1}.

Let $E=\{1,\dots,m\}$. It is easy to see that $V_\alpha$ has the following form
\begin{equation}\label{opcc}
V_\alpha: \left\{\begin{array}{lll}
               x'_{k}= \alpha x_k^2 +
               \frac{2\alpha}{m-2}\displaystyle\sum_{\substack{i,j\in E\setminus\{k\}\\ i<j}} x_ix_j + 2(1-\alpha)x_m x_{\pi(k)}, \ \ 1\leq k\leq m-1\\[5mm]
               x'_{m}= x_m^2+
               \frac{2\alpha}{m-2}\displaystyle\sum_{\substack{i,j=1\\ i<j}}^{m-1} x_ix_j +
               (1-\alpha) \Big(\displaystyle\sum_{\substack{i=1}}^{m-1} x_i\Big)^2.
\end{array}\right.
\end{equation}

Since $V_0$ is a quasi strictly non-Volterra QSO \eqref{op1}  and $V_1$ is a non-Volterra QSO \eqref{eqso}, hereafter we will  consider the case $\alpha\in (0,1)$.

\begin{proposition} The vertex $\mathbf{e}_m=(0,\dots,0,1)$ and the point $\mathbf{x}^*=\big(x_1^*,\dots, x_m^*\big)\in S^{m-1}$, \[x_1^*=\cdots=x_{m-1}^* =\frac{1}{ (2-\alpha)(m-1)+\alpha}, \quad x_m^*=\frac{(1-\alpha)(m-1)+\alpha}{ (2-\alpha)(m-1)+\alpha},\]
 are fixed points of  the operator \eqref{opcc} for any $\alpha\in (0,1)$.
\end{proposition}

\begin{proof}
The equation $V_\alpha(\mathbf{x})=\mathbf{x}$ has the form
\begin{equation}\label{eqcc}
 \left\{\begin{array}{lll}
               x_{k}= \alpha x_k^2 +
               \frac{2\alpha}{m-2}\displaystyle\sum_{\substack{i,j\in E\setminus\{k\}\\ i<j}} x_ix_j + 2(1-\alpha)x_m x_{\pi(k)}, \ \ 1\leq k\leq m-1\\[5mm]
               x_{m}= x_m^2+
               \frac{2\alpha}{m-2}\displaystyle\sum_{\substack{i,j=1\\ i<j}}^{m-1} x_ix_j +
               (1-\alpha) \Big(\displaystyle\sum_{\substack{i=1}}^{m-1} x_i\Big)^2.
\end{array}\right.
\end{equation}

From the first $(m-1)$ equations of \eqref{eqcc}, we get the following system of equations
\begin{align}\label{eqdef}
2(1-\alpha)x_m\big(x_{\pi(1)}-x_{\pi(2)}\big)&=\big(x_1-x_2\big)\cdot\big(1+ (m-3)\alpha -(m-2)\alpha(x_1+x_2) \big)\notag\\
2(1-\alpha)x_m\big(x_{\pi(2)}-x_{\pi(3)}\big)&=\big(x_2-x_3\big)\cdot\big(1+ (m-3)\alpha -(m-2)\alpha(x_2+x_3) \big)\notag\\
               \ \ \ &\vdots \ \ \  \\
2(1-\alpha)x_m\big(x_{\pi(m-1)}-x_{\pi(1)}\big) &=\big(x_{m-1}-x_1\big)\cdot\big(1+ (m-3)\alpha -(m-2)\alpha(x_1+x_{m-1}) \big). \notag
\end{align}

Evidently,   if $x_m=1$ then we get the vertex $\mathbf{e}_m$ and  it is easy to see that $x_m=0$ does not satisfy the last equation of \eqref{eqcc}.
So $0<x_m<1$ and it follows that  system \eqref{eqdef} has the unique solution
\[x_1=x_2=\cdots=x_{m-1},\]
and substituting in the last equation of \eqref{eqcc} one has
\[x_m= x_m^2+ (m-1)\alpha x_1^2+(1-\alpha)(m-1)^2x_1^2.\]
Using $x_m=1-(m-1)x_1$, we get
\begin{equation}\label{eqx1}
\big((2-\alpha)(m-1)+\alpha\big)x_1^2-x_1=0.
\end{equation}

Equation \eqref{eqx1} has the following solutions
\[x_1^{(1)}=0 \quad \text{and} \quad x_1^{(2)}=\frac{1}{ (2-\alpha)(m-1)+\alpha}.\]

If we take $x_1=0$ then we get the fixed point $x_m=1$, that is the vertex  $\mathbf{e}_m$.
Let us take $x_1^*=x_1^{(2)}$. Then it follows
\[x_m^*=\frac{(1-\alpha)(m-1)+\alpha}{ (2-\alpha)(m-1)+\alpha}.\]

\end{proof}

Consider the following function
\[f_\alpha (x)=\Big(2-\frac{(m-2)\alpha}{ m-1}\Big) x^2-2\Big(1-\frac{(m-2)\alpha}{ m-1}\Big)x+1-\frac{(m-2)\alpha}{ m-1}.\]

\begin{definition} Let $f \colon A \to A$ and $g \colon B \to B$ be two maps. $f$ and $g$ are called
topologically conjugate if there exists a homeomorphism $h \colon A \to B$ such that $h\circ f = g\circ h$.
\end{definition}
Note that those mappings which are topologically conjugate are completely equivalent
in terms of their dynamics. In particular, $h$ gives a one-to-one correspondence between
the set of limit points of $f$ and $g$.

\begin{lemma}\label{falfa} If $0<x<1$, then
\[\lim\limits_{n\rightarrow\infty} f^n_\alpha(x)=x_m^*.\]
\end{lemma}

\begin{proof}

Consider the function $f_\alpha(x)$. Taking
\[h(x) = \frac{(m-2)\alpha-2(m-1)}{ 2(m-1)}x + \frac{3(m-1)-2(m-2)\alpha}{ 4(m-1)},\]
one can see that the function $f_\alpha(x)$
is topologically conjugate to the well-known logistic map $g(x) = \mu x(1-x)$ with $\mu=2$.

 It is known that if $\mu\in (1, 3)$ then  the logistic function $g(x)$ has an attracting fixed point $x_\mu = (\mu-1)/\mu$ and a repelling fixed
point $x = 0$. All trajectories (except when started at the fixed point) will approach
fixed point $x_\mu$ (see \cite{Dev}).

From this fact, by the conjugacy argument, it follows that all trajectories of $f_\alpha(x)$ started an initial point will approach fixed
point $x_m^*$.

The lemma is proven.

\end{proof}

\begin{theorem} For the QSO $V_\alpha$ the following statements are true:
\begin{itemize}
\item[i)] The function $\phi(\mathbf{x}) =x_m$ is a Lyapunov function for the operator $V_\alpha$;
\item[ii)] The set of limit points is $\omega_{V_\alpha}\big(\mathbf{x}^{(0)}\big)=\{\mathbf{x}^*\}$ for any $\mathbf{x}^{(0)}\in S^{m-1}$.
\end{itemize}
\end{theorem}

\begin{proof}  i) Evidently the function $\phi(V_\alpha(\mathbf{x}))=x'_m$ is continuous on $S^{m-1}$ and
\[\min\limits_{\mathbf{x}\in S^{m-1}} \phi(V_\alpha(\mathbf{x})) =x_m^* \ \ \text{iff} \ \  \mathbf{x}=\mathbf{x}^*, \ \ \text{where}\]
\[x_1^*=\cdots=x_{m-1}^* =\frac{1}{ (2-\alpha)(m-1)+\alpha}, \ \ x_m^*=\frac{(1-\alpha)(m-1)+\alpha}{ (2-\alpha)(m-1)+\alpha}.\]

Using the well-known Maclaurin's inequality, we have
\[\frac{2\alpha}{m-2} \displaystyle\sum_{\substack{i,j=1\\ i<j}}^{m-1} x_ix_j =
(m-1)\alpha \frac{\displaystyle\sum_{\substack{i,j=1\\ i<j}}^{m-1} x_ix_j}{
\binom{2}{m-1}} \leq \frac{\alpha}{ m-1} (x_1+\dots+x_{m-1})^2.\]

Using the last and \eqref{opcc} one has that
\begin{align*}
\phi(V_\alpha(\mathbf{x}))-\phi(\mathbf{x})&=x'_m-x_m= x_m^2+
               \frac{2\alpha}{m-2}\displaystyle\sum_{\substack{i,j=1\\ i<j}}^{m-1} x_ix_j +
               (1-\alpha) \Big(\displaystyle\sum_{\substack{i=1}}^{m-1} x_i\Big)^2-x_m\\
&= x_m^2+(1-\alpha) (1-x_m)^2 +  \frac{2\alpha}{m-2}\displaystyle\sum_{\substack{i,j=1\\ i<j}}^{m-1} x_ix_j-x_m\\
&\leq x_m^2+(1-\alpha) (1-x_m)^2 +  \frac{\alpha}{m-1}(1-x_m)^2-x_m\\
&=x_m^2+\Big(1-\frac{(m-2)\alpha}{m-1}\Big) (1-x_m)^2-x_m\\
&= f_\alpha(x_m)-x_m,
\end{align*}

Some analysis gives us
\[\max  \Big(f_\alpha(x_m)-x_m\Big) =0, \ \ x_m\in \Big[x^*_m,1\Big]. \]

Thus $\phi(V_\alpha(\mathbf{x}))\leq \phi(\mathbf{x}) $ for all $\mathbf{x}\in S^{m-1}$, that is
the function $\phi(\mathbf{x})$ is a Lyapunov function for the operator $V_\alpha$.

Consequently, it  holds
\[\phi(\mathbf{x}^{(n+1)}) \leq \phi(\mathbf{x}^{(n)}), \quad n=0,1,2,\dots   \]
and  due to Lemma \ref{falfa}  we have
\begin{equation}\label{mcorr}
x_m^* \leq \lim\limits_{n\rightarrow\infty} \phi(\mathbf{x}^{(n+1)})
\leq \lim\limits_{n\rightarrow\infty} f^n_\alpha(x_m)=x_m^*.
\end{equation}

ii) Let $\mathbf{x}^{(0)}\in S^{m-1}\setminus \{\mathbf{e}_m\}$. Due to claim of  part i) we have $ \lim\limits_{n\rightarrow\infty} x_m^{(n)}=x_m^*<1/2$. So for enough large $n$ it  holds
\[\mathbf{x}^{(n)}\in S^{m-1}_{1/2}=\{\mathbf{x}^{(n)}\in S^{m-1}: x_m^{(n)}<\frac{1}{2}\}. \]

Using induction for $u\neq v,  \ u,v=1,\dots,m-1$, and $s=\lcm \big(\ord (\tau_1),\dots, \ord (\tau_q)\big)$, one has
\begin{align*}
x_u^{(s)}-x_v^{(s)}&= \big(x_u-x_v\big)\Big(\Big(\frac{\alpha}{ m-2}\Big)^s\prod\limits_{j=0}^{s-1} \big(m(x_u^{(j)}+x_v^{(j)}) - 2\big)
    + 2^s(1-\alpha)\prod\limits_{j=0}^{s-1}x_m^{(j)} \Big) \\
\end{align*}

For $\mathbf{x}\in S^{m-1}_{1/2}$ we have
\begin{align*}
\big|x_u^{(s)}-x_v^{(s)}\big|&=\big|x_u-x_v\big|\bigg| \Big(\frac{\alpha}{m-2}\Big)^s\prod\limits_{j=0}^{s-1} \big(m(x_u^{(j)}+x_v^{(j)}) - 2\big)
    + 2^s(1-\alpha)\prod\limits_{j=0}^{s-1}x_m^{(j)}\bigg|\\
    &\leq \big|x_u-x_v\big| \bigg| \Big(\frac{\alpha}{m-2}\Big)^s\prod\limits_{j=0}^{s-1} \big|m(x_u^{(j)}+x_v^{(j)}) - 2\big|+
     (1-\alpha)\bigg|\\
    &\leq  \big|x_u-x_v\big| \big(1-\alpha+\alpha^s\big).
\end{align*}
Therefore for $u\neq v, \  u,v=1,\dots,m-1$, it follows
\begin{equation}\label{1coor}
\lim\limits_{n\rightarrow\infty} \big|x_u^{(s\cdot n)}-x_v^{(s\cdot n)}\big|\leq \big|x_u-x_v\big| \lim\limits_{n\rightarrow\infty} \big(1-\alpha+\alpha^s\big)^n=0.
\end{equation}

Consequently, from \eqref{mcorr} and \eqref{1coor}, it follows that for the operator $V_\alpha^s$ the trajectory  converges to
the fixed point $\mathbf{x}^*$.

Let $\mathbf{x}^{(0)}\in S^{m-1}\setminus \{\mathbf{e}_m\}$ be an initial point and the sequence  $\big\{\mathbf{x}^{(n)}\big\}$ its trajectory.
As shown in above if $n=s\cdot k$ the subsequence  $\{\mathbf{x}^{(k\cdot s)}\}$ converges to $\mathbf{x}^*$  when $k\rightarrow\infty$.

Since the point $\mathbf{x}^*$ is a fixed point and the operator $V_\alpha$ is continuous
for the cases $n=s\cdot k+l, \  l=1,\dots,s-1$, we also have
\[\lim\limits_{n\rightarrow\infty} V^l\big(\mathbf{x}^{(k\cdot s)}\big) = \mathbf{x}^*.\]

Thus, the trajectory starting at any point $\mathbf{x}^{(0)}\in S^{m-1}$ converges to
the fixed point $\mathbf{x}^*$.
\end{proof}

\section*{Acknowledgments}

This work was partially supported by a grant from the IMU-CDC.
The first author (UJ) thanks the University of Santiago de Compostela (USC), Spain,  for the kind hospitality and for  providing all facilities.
The authors were partially supported by Agencia Estatal de Investigaci\'on (Spain), grant MTM2016-79661-P and by Xunta
de Galicia, grant ED431C 2019/10 (European FEDER support included, UE).
%\bibliographystyle{elsart-num-sort}
%\bibliography{biblio8}

\end{document}